\documentclass[12pt]{amsart}

\usepackage{amsmath, amscd, amssymb, euler}
\usepackage[frame,cmtip,arrow,matrix,line,graph,curve]{xy}
\usepackage{graphpap, color,bbm, verbatim}
\usepackage[mathscr]{eucal}
\usepackage[normalem]{ulem}
\usepackage{pgf,tikz}
\usetikzlibrary{arrows}

\numberwithin{equation}{section}

\newcommand{\CC}{\mathbb{C}}

\newcommand{\PP}{\mathbb{P}}

\newcommand{\bB}{\mathbf{B}}

\newcommand{\bM}{\mathbf{M}}

\newcommand{\cal}{\mathcal}

\def\cF{{\cal F}}

\def\cO{{\cal O}}

\newcommand{\ses}[3]{0\lr{#1}\lr{#2}\lr{#3}\lr 0}

\def\lr{\rightarrow}

\input xy
\xyoption{all}

\DeclareMathOperator{\Ext}{Ext} 
\DeclareMathOperator{\Hom}{Hom} 

\newtheorem{prop}{Proposition}[section]
\newtheorem{theo}[prop]{Theorem}

\newtheorem{lem}[prop]{Lemma}
\newtheorem{coro}[prop]{Corollary}

\theoremstyle{definition}

\newtheorem{exam}[prop]{Example}
\newtheorem{defi}[prop]{Definition}
\newtheorem{rema}[prop]{Remark}

\newcommand{\opt}{\mathcal{O}_{\mathbb{P}^2}}
\newcommand{\pt}{{\mathbb{P}^2}}
\newcommand{\pp}{\mathbb{P}}

\newcommand{\ds}[1][\chi]{\mathcal{C}_{d,#1}}

\def\mapright#1{\,\smash{\mathop{\rightarrow}\limits^{#1}}\,}

\title{Cohomology bounds for sheaves of dimension one}

\author{Jinwon Choi}
\address{Department of Mathematics, University of Illinois at Urbana-Champaign, 1409 E Green St., Urbana, IL 61801, United States}
\email{choi29@illinois.edu}

\author{Kiryong Chung}
\address{School of Mathematics, Korea Institute for Advanced Study, Seoul 130-722, Korea}
\email{krjung@kias.re.kr}

\begin{document}
\begin{abstract}
 We find the sharp bounds on $h^0(F)$ for one-dimensional semistable sheaves $F$ on a projective variety $X$ by using the spectrum of semistable sheaves. The result generalizes the Clifford theorem. When $X$ is the projective plane $\mathbb{P}^2$, we study the stratification of the moduli space by the spectrum of sheaves. We show that the deepest stratum is isomorphic to a subscheme of a relative Hilbert scheme. This provides an example of a family of semistable sheaves having the biggest dimensional global section space.
\end{abstract}

\maketitle

\section{Introduction and the results of the paper}
\subsection{Motivations and the main theorem}
In the study of moduli spaces of semistable sheaves on the projective variety, it is useful to know the upper bounds on the dimensions of the cohomology groups of the semistable sheaves with fixed Hilbert polynomial.
This is essential for a classification of semistable sheaves with respect to the cohomological conditions (For example, see \cite{maican,maican1,maican2} and \cite{fredu}). It is also very helpful for analyzing the forgetting map from the moduli space of pairs to that of semistable sheaves (For definitions and examples, see \cite[\S4]{kiryong}).

Historically, C. Simpson constructed the moduli spaces of semistable sheaves as a compactification of the moduli space of Higgs bundles on a variety (\cite{simp}). The moduli space of Higgs bundles has also been studied by algebraic geometers and physicists regarding Hamitonian systems etc. As a natural generalization, we may consider \emph{twisted} Higgs bundles by using a general line bundle instead of the cotangent bundle. The moduli spaces of twisted Higgs bundles has been studied widely for its geometric structure (\cite{rayan,moz}). In this paper, we are interested in cohomology bounds for semistable sheaves on a projective variety. It turns out that the global section space has the maximal dimension if the sheaf is a plane sheaf, which can be identified with a twisted Higgs bundle on $\PP^1$ twisted by $\cO(1)$.

The \emph{spectrum} of a one-dimensional sheaf on $\PP^2$ is the sequence of degrees in a decomposition of the corresponding Higgs bundle into line bundles on $\PP^1$ (Definition \ref{defofspectrum}). We study the stratification of the moduli space of semistable sheaves with respect to the spectrum. By classifying all possible spectra, we prove a conjecture on the cohomology bounds of the sheaves suggested in \cite{maican}. More precisely, let $\bM(d,\chi)$ be the moduli space of semistable sheaves on $\PP^2$ with Hilbert polynomial $dm+\chi$.
Let $g(d):= \frac{(d-1)(d-2)}{2}$ be the arithmetic genus of the degree $d$ plane curve. We prove that

\begin{theo}\label{maintheorem}
Let $F$ be a semistable sheaf in $\bM(d,\chi)$.
\begin{enumerate}
  \item If $\chi\ge g(d)$, then $h^0(F)=\chi$.

  \item Suppose $\chi<g(d)$ and write $\chi+\frac{d(d-3)}{2}=kd+r$ with $0\le r< d$. Then,
\begin{equation}\label{eq:max} h^0(F)\le \frac{(k+2)(k+1)}{2} +\mathrm{max}\{0, k-d+r+2\}.\end{equation}
Furthermore, there are families of semistable sheaves in $\bM(d,\chi)$ for which the equality holds in \eqref{eq:max}.
\end{enumerate}
\end{theo}

\begin{rema}
Theorem \ref{maintheorem} is a generalization of R. Hartshorne's result \cite{har}. He proved the same bounds for sheaves supported on a irreducible plane curve by using induction on the degrees of the curves. In this paper, we give another proof which can easily be generalized to a general projective variety. It is interesting that there exist sheaves supported on irreducible curves with maximal global section space.
\end{rema}

The proof of Theorem \ref{maintheorem} goes as follows. Part (1) is established in Lemma \ref{lem:h1vanishing}.
For part (2), we associate a semistable Higgs bundle $(G,G \mapright{\phi} G(1))$ on $\PP^1$ (in the sense of $\cO_{\PP^1}(1)$-twisting)
with a semistable sheaf $F$ on $\PP^2$, where $G=\pi_*F$ is the direct image sheaf of $F$ by the projection $\pi:\PP^2-\{a\}\lr \PP^1$ from a point $a$ $(\notin \mbox{Supp}(F)$).
Since $F$ is pure, $G$ is locally free and hence it decomposes into a direct sum of line bundles on $\PP^1$. For the sheaf $F$ to be semistable, the degrees of these line bundles must satisfy certain numerical conditions (Proposition \ref{balanced}), which enables us to determine the upper bounds on $h^0(F)$ by a combinatorial reasoning (Lemma \ref{lem:maxspectrum}).

The same technique can be applied to a general projective variety. Let $X\subset \PP^r$ be a projective variety with a fixed polarization $\cO_X(1)$.
We consider a projection from an $(r-2)$-dimensional hyperplane away from the support of a semistable sheaf $F$. Similarly as before, we may define the \emph{generalized spectrum} of $F$.
By studying its combinatorial property, Theorem \ref{maintheorem} can easily be generalized to $X$. For detail, see \S\ref{sec:general}.
\begin{theo}\label{mainthm2}
Let $F$ be a semistable sheaf on a projective variety $X$ with Hilbert polynomial $dm+\chi$. Then the dimension $h^0(F)$ has the same upper bound as in the case of $\PP^2$.
\end{theo}

As a corollary, we prove the ``generalized Clifford theorem'' (Corollary \ref{conj1}) which is a generalization of a conjecture in \cite[\S1.4]{maican}.

When $X=\PP^2$, we also study the stratification of the moduli space by the spectrum of sheaves (\S\ref{sec:deep}). It turns out that our choice of the spectrum that gives the upper bounds on $h^0(F)$ corresponds to the deepest stratum in this stratification, in the sense that it has the biggest codimension among all strata. We identify this stratum with a subscheme of a relative Hilbert scheme. In particular, this proves the bounds in Theorem \ref{maintheorem} are sharp bounds. However, we do not know whether the upper bounds in Theorem \ref{maintheorem} are achieved or not for a general $X$ not containing a projective plane $\PP^2$.

\medskip
\textbf{Acknowledgement.}
We would like to thank Sheldon Katz and Young-Hoon Kiem for helpful discussions and comments. Specially, S. Katz's comments on Proposition \ref{katz} are very helpful. We would also like to thank JiUng Choi
for constructing a computer program to find all \emph{maximizing spectra}.

\section{The proof of Theorem \ref{maintheorem}}\label{sec2}
In this section, we prove the cohomology bounds for plane sheaves. The following is from \cite[Lemma 4.2.4]{choi}.
\begin{lem}\label{lem:h1vanishing}
Let $F$ be a semistable sheaf on $\pt$ with Hilbert polynomial $dm+\chi$. Then $H^1(F)=0$ if $\chi \ge g(d)$.
\end{lem}
\begin{proof}
We know that $F$ is supported on some degree $e$ Cohen-Macaulay curve $C$ in $\pt$ where $1\le e\le d$. By adjunction formula, we have $\omega_C\simeq \cO_C(e-3)$. By applying Serre duality on $C$, we have $H^1(F)^*\simeq \Hom(F,\cO_C(e-3))$. Suppose there is a nonzero map: $F\to \cO_C(e-3)$. Then, by semistability of $F$ and $\cO_C$, we have
\begin{equation}\label{eq:ce3}\mu(F)\le \mu(\cO_C(e-3)).\end{equation}
Since the Hilbert polynomial $\chi(\cO_C(e-3)(m))=em+\frac{e(e-3)}{2}$, we have $\mu(\cO_C(e-3))=\frac{e-3}{2}.$

Then by \eqref{eq:ce3},
\[ \frac{\chi}{d} \le \frac{e-3}{2} \le \frac{d-3}{2}.\]
Therefore, if $\chi \ge \frac{d(d-3)}{2}+1= \frac{(d-1)(d-2)}{2}$, then $H^1(F)=0$.
\end{proof}
To prove main theorem, let us recall the notion of the \emph{spectrum} of a pure sheaf on $\PP^2$ (\cite{lepot1}).
Let $\pi:\PP^2-\{a\} \lr \PP^1$ be the projection map from a point $a\in \PP^2$. For a pure sheaf $F$ with $a \notin \mbox{Supp}(F)$,
the direct image sheaf $G:=\pi_*F$ is a locally free sheaf on $\PP^1$.
Note that $\PP^2\setminus\{a\}=\mbox{Tot}(\cO_{\PP^1}(1))=\mathrm{Spec(Sym}\cO_{\PP^1}(-1))$ and $\pi$ is the affine morphism $\mathrm{Spec(Sym}\cO_{\PP^1}(-1))\to \PP^1 $. Hence, the locally free sheaf $G$ on $\PP^1$ has a natural $\cO_{\PP^1}(-1)$-module structure or equivalently, it admits a sheaf homomorphism $\phi: G\lr G(1)$.

\begin{rema}
The pair
$
(G, G\mapright{\phi} G(1))
$
is a \emph{twisted Higgs bundle} on $\PP^1$. In \cite{lepot1}, it is shown that a sheaf is (semi)stable if and only if the associated twisted Higgs bundle is (semi)stable. Let $U_a$ be the open subscheme of $\bM(d,\chi)$ consisting of sheaves whose support does not pass through a fixed point $a\in \PP^2$ (cf. \cite{lepot1}). It follows that $U_a$ is isomorphic to the moduli space of twisted Higgs bundles.
\end{rema}

\begin{defi}\label{defofspectrum}
Under the above notations and assumptions, $$G=\pi_*F\cong\oplus_{i=1}^{d} \cO_{\pp^1}(a_i)$$ is a locally free sheaf on $\PP^1$ of rank $d$ if the Hilbert polynomial of $F$ is $\chi(F(m))=dm+\chi$. Let us define the degree sequence $$v=[a_1,a_2,\cdots,a_d]$$where $a_1\ge \cdots\ge a_d$ to be the \emph{spectrum} of the sheaf $F$.
\end{defi}

\begin{prop}\label{balanced}
Let $v=[a_1,a_2,\cdots,a_d]$ be the spectrum of a semistable sheaf $F$ in $\bM(d,\chi)$. Then it must satisfy two conditions:
\begin{enumerate}
  \item $\sum_{i=1}^{d}a_i=\chi-d$ and
  \item $a_j-a_{j+1} \le 1$ for all $1\leq j\leq d-1$ (\emph{Balanced property}).
\end{enumerate}
\end{prop}
\begin{proof}
Part (1) is obvious from the condition $\chi(F)=\chi$. Part (2) is equivalent to the semistability of $F$, see the proof of \cite[Lemma 3.12]{lepot1}. We will prove this again in more general situation in \S\ref{sec:general}.
\end{proof}

Note that if $\pi_*F=\oplus_{i=1}^{d} \cO_{\pp^1}(a_i)$, $$h^0(F):=\mbox{dim}H^0(\PP^2,F)= \sum_{a_k\geq 0}(a_k+1)$$ since the map $\pi$ is an affine map.
\begin{rema}\label{rem1}
To determine $[a_1,\cdots, a_d]$, it is enough to know $k\mapsto h^0(F(k))$ or equivalently $k\mapsto h^1(F(k))$. In particular, the spectrum of a sheaf is independent of the choice of the center of the projection $\pi$.
\end{rema}

\begin{exam}\label{lem:dualspectrum}
\begin{enumerate}
  \item   Let $C$ be a curve of degree $d$ in $\PP^2$. Then the spectrum of the structure sheaf $\cO_C$ is
  \[ [0,-1, -2,\cdots, 1-d].\]
Indeed, for $0\le k \le d-1$, we have $h^0(\cO_C(k))=\frac{(k+1)(k+2)}{2}$ and the above spectrum is the unique one satisfying this condition.
\item  Let $[a_1,\cdots, a_d]$ be the spectrum of a semistable sheaf $F$ with multiplicity $d$. Then the spectrum of $F(k)$ is  $[a_1+k, \cdots, a_d+k]$ and the spectrum of $F^D:=\mathcal{E} xt^1(F,\omega_{\PP^2})$ is $[-2-a_d, \cdots, -2-a_1]$. In fact, the spectrum of $F(k)$ is straightforward by the projection formula $\pi_*F(k)=\pi_*F\otimes \cO_{\PP^1}(k)$.
  For $F^D$, by \cite[Proposition 5]{maicandual} or \cite[Proposition 4.2.8]{choi}, we have $h^0(F^D(-k))=h^1(F(k))$ for any integer $k$.
  By a straightforward induction, one can prove the claim.
\item If a spectrum $v$ satisfies the conditions in Proposition \ref{balanced}, there exists a semistable sheaf whose spectrum is $v$. In fact, we can explicitly construct examples of \emph{torus equivariant} semistable sheaves in each spectrum by the method of \cite{choip1}. (cf. Example \ref{example1})
\end{enumerate}
\end{exam}
\begin{lem}\label{lem:maxspectrum}
Let $[a_1,\cdots, a_d]$ be the spectrum of $F$ in $\bM(d,\chi)$. If $a_j=a_{j+1}$ for at most one $j$, then $h^0(F)$ is maximal among all sheaves in $\bM(d,\chi)$.
\end{lem}
\begin{proof}
Suppose not. Let $[a_1,\cdots, a_d]$ be the spectrum for a sheaf $F$ in $\bM(d,\chi)$ having bigger $h^0(F)$. Suppose first that we have $a_{j-1}> a_j=a_{j+1}=\cdots=a_{j+k}>a_{j+k+1}$ for some $j$ and $k\ge 2$. Consider the spectrum $$[\cdots, a_{j-1},a_j+1, a_{j+1}, \cdots, a_{j+k-1}, a_{j+k}-1,a_{j+k+1},\cdots].$$ It is easy to see that the Euler characteristic does not change while $h^0(F)$ either increases or remains the same. Hence we may assume the same number is not repeated more than twice in $[a_1,\cdots, a_d]$.

Now we suppose that $a_j=a_{j+1}$ for at least two $j$. Take $a_{j_1}=a_{j_1+1}$ and $a_{j_2}=a_{j_2+1}$, $j_1<j_2$. We can choose $[a_1,\cdots, a_d]$ so that $j_2-j_1$ is minimal. Consider the spectrum
\[ [a_1, \cdots, a_{j_1}+1,a_{j_1+1}, \cdots,  a_{j_2},a_{j_2+1}-1, \cdots, a_d].  \]
This spectrum has the same Euler characteristic, but $h^0(F)$ either increases or remains the same. This contradicts the minimality of  $j_2-j_1$.
\end{proof}

As we will see in the following examples, when we fix $d$ and $\chi$, the spectrum of a sheaf whose global section space is maximal need not be unique.
\begin{exam}\label{ex:ds}
\begin{enumerate}
  \item There is unique spectrum satisfying the condition in Lemma \ref{lem:maxspectrum} for a given $(d,\chi)$. We present here a few examples of such spectra for the convenience of readers.
      \begin{center}
      \begin{tabular}{|c|l|}
        \hline
        $(d,\chi)$ &  \\
        \hline
        $(5,0)$ & $[1,0,-1,-2,-3]$ \\
        $(5,1)$ & $[1,0,-1,-2,-2]$ \\
        $(5,2)$ & $[1,0,-1,-1,-2]$ \\
        $(5,3)$ & $[1,0,0,-1,-2]$ \\
        \hline
        \hline
        $(6,3)$ & $[2,1,0,-1,-2,-3]$ \\
        $(6,4)$ & $[2,1,0,-1,-2,-2]$ \\
        $(6,5)$ & $[2,1,0,-1,-1,-2]$ \\
        $(6,6)$ & $[2,1,0,0,-1,-2]$ \\
        \hline
      \end{tabular}
      \end{center}
  When $(d,\chi)$ is of the form $(2k+1,0)$ or $(2k,k)$, the corresponding spectrum is $$[k-1,k-2,\cdots, k-d].$$ If $(d,\chi)$ is of the form $(2k+1,s)$ or $(2k,k+s)$ for $0\le s\le d$, last $s$ terms are increased by one from $[k-1,k-2,\cdots, k-d]$. Similar rules can be found for $s>d$.

  \item Note that the converse of Lemma \ref{lem:maxspectrum} is not true, that is, there may be different types of spectrum having the largest global section space. For example, one can check that there are three possible spectra for $(d,\chi)=(6,0)$:
  $$
\hspace{3em}  \{[0,0,0,-1,-2,-3],\mbox{ } [1,0,-1,-1,-2,-3],\mbox{ } [1,0,-1,-2,-2,-2]\}.
  $$
  For the corresponding loci of semistable sheaves in $\bM(6,0)$, see the Table 4 of \cite{maican2}. We remark that the locus corresponding to the spectrum $[1,0,-1,-1,-2,-3]$ has the biggest codimension. We will focus on such loci in \S3.
\end{enumerate}
\end{exam}

Returning to our main claim, we prove the part (2) of Theorem \ref{maintheorem}.
\begin{proof}[Proof of Theorem \ref{maintheorem}]
It remains to compute $h^0(F)$ of the sheaf $F$ having the spectrum satisfying the condition in Lemma \ref{lem:maxspectrum}. We will show the existence of such sheaves later in Example \ref{example1} and \S 3.

Such spectrum is completely determined by $a_1$ and $1\le j\le d$ such that $a_j=a_{j+1}$.
Here, $j=d$ means there is no $j$ with $a_j=a_{j+1}$. For notational convenience, we let $k:=a_1$ and $r:=d-j$, $0\le r\le d-1$.

Then it is easy to see that
\begin{equation}\label{eq:sumofspec}
a_1+\cdots+a_d= kd - \frac{d(d-1)}{2} +r.
\end{equation}
So, $\chi = kd - \frac{d(d-3)}{2} +r$ and hence $k$ and $r$ are uniquely determined by $\chi$.

Suppose $k\ge d-3$ and $r\ge 1$, in other words, $\chi\ge \frac{d(d-3)}{2}+1=g(d)$. Then we see that $a_d\ge -1$, which implies that all higher cohomologies vanish and we have $h^0(F)=\chi$. This gives another proof of Lemma \ref{lem:h1vanishing}.

Now we suppose $\chi< g(d)$. Then the spectrum is given by
\begin{equation}\label{eq:maxspec}
[k, k-1, \cdots, k-d+r+1, k-d+r+1, \cdots, k-d+2].
\end{equation}

If $r\le d-k-1$, the nonnegative terms in the spectrum are $(k, k-1, \cdots, 0)$ and hence we have
\[ h^0(F) =\frac{(k+2)(k+1)}{2}.  \]

On the other hand, if $r> d-k-1$ we have 
\[ h^0(F) =\frac{(k+2)(k+1)}{2} + (k-d+r+2).  \]
So the theorem follows.
\end{proof}

The following example shows that the bounds in Theorem \ref{maintheorem} are sharp.

\begin{exam}\label{example1}
  Let $k$, $j$, and $r$ be determined by $d$ and $\chi$ as in the proof of Theorem \ref{maintheorem}. Let us denote by $x$, $y$, and $z$ the homogeneous coordinates for $\PP^2$. Let $C_d$ be the $d$-fold thickening of a fixed line in $\PP^2$. For example, Let $C_d$ be the subscheme defined by the ideal $\langle z^d \rangle$.

Let $Z_j$ be the subscheme of $C_d$ defined by the ideal $\langle x, z^{j}\rangle$. We take $F=I_{Z_j,C_d}(k+1)$ be the twisted ideal sheaf of $Z_j$ in $C_d$. Then we claim that $F$ is a semistable sheaf in $\bM(d,\chi)$ whose spectrum is \eqref{eq:maxspec}.

  Since $\pt$ is a toric variety, we consider the natural action of the torus $T=(\CC^*)^2$. Then since the ideal defining $C$ and $Z_r$ is torus invariant, we see that $F$ is a $T$-equivariant sheaf. It is known that a $T$-equivariant sheaf is semistable if and only if all of its $T$-equivariant subsheaves satisfy the slope condition \cite[Proposition 3.19]{kool}. One can check that all saturated $T$-equivariant subsheaves of $F$ are of the form $I_{Z_{j'},C_{d'}}(k+1)$, where $d'\le d$ and $j'\le j$. By computing slopes, we see that $F$ is semistable. For more details on the torus equivariant sheaves, see \cite[\S 2.3]{choi} and \cite{choip1}.

  It remains to check that the spectrum of $F$ is \eqref{eq:maxspec}. When we restrict $F$ to the subscheme $C_j$ defined by the ideal $\langle z^j \rangle$, we have $$F|_{C_j}= I_{Z_j,C_j}(k+1) \simeq \cO_{C_j}(k),$$
  and the kernel of the restriction map is $\cO_{C_{d-j}}(k-j+1)$. So we have an exact sequence
  \[\ses{\cO_{C_{d-j}}(k-j+1)}{F}{ \cO_{C_j}(k)}. \]
  Since $r=d-j$, so the spectrum of $F$ is \eqref{eq:maxspec}.
\end{exam}

\section{Upper bound for a projective variety}
\label{sec:general}
In this section, we define a generalized spectrum of sheaves on a projective variety. This is a natural generalization of the $\PP^2$ case. Using this we prove Theorem \ref{mainthm2}.
Let $F$ be a coherent sheaf on a projective variety $X$ with a fixed embedding $i:X\subset \PP^r$. Since $H^0(X,F)=H^0(\PP^r, i_*F)$, we can regard $F$ as a sheaf on a projective space $\PP^r$.
Let the Hilbert polynomial of $F$ be $\chi(F(m))=dm+\chi$. Let $F$ be supported on a curve $C \subset \PP^r$. For an $(r-2)$-dimensional linear subspace $H$ such that $H\cap C= \emptyset$, let
$$
\pi:\PP^r\setminus H \longrightarrow \PP^1
$$
be the projection from $H$. Since $\pi_*F$ is a locally free sheaf of rank $d$ on $\PP^1$, we have $\pi_*F=\cO_{\pp^1}(a_1)\oplus \cO_{\pp^1}(a_2)\oplus \cdots \oplus\cO_{\pp^1}(a_d)$.
\begin{defi}
Under above notation, let us define
$$
v=[a_1,a_2,\cdots, a_d]
$$
where $a_1\ge \cdots\ge a_d$ to be the \emph{generalized spectrum} of the sheaf $F$.
\end{defi}
We are ready to prove Theorem \ref{mainthm2}.
\begin{proof}[Proof of Theorem \ref{mainthm2}]
It is enough to prove that the generalized spectrum of a semistable sheaf $F$ satisfies the two conditions in Proposition \ref{balanced}. Note that $\PP^r\setminus H$ is the total space of the bundle $\cO_{\pp^1}^{\oplus (r-1)}$ and $\pi$ is the affine morphism
\[\pi\colon \mathrm{Spec(Sym}\cO_{\PP^1}(-1)^{\oplus (r-1)})\to \PP^1.\]
Hence, part (1) of Proposition \ref{balanced} must be satisfied. For part (2), we use the semistability of $F$. Suppose the generalized spectrum of $F$ does not satisfy the condition in part (2). Then we can write $\pi_*F= G'\oplus G''$ such that $\Hom(G', G''(1))=0$. Hence $G'$ is a $\mathrm{Sym}\cO_{\PP^1}(-1)^{\oplus (r-1)}$-submodule of $\pi_*F$. So there exists a subsheaf $F'$ of $F$ induced by $G'$. It is easy to check that the slope of $F'$ is greater than the slope of $F$, which violates the semistability.
\end{proof}


We now prove a generalization of the Clifford theorem.
\begin{coro}\label{conj1}
For a semistable sheaf $F$ on a projective variety $X$ with Hilbert polynomial $dm+\chi$ with $\chi \ge 1-g(d)$ and $h^1(F)>0$, we have
\begin{equation}\label{clifford}
h^0(F)\le 1+ \frac{\chi}{2} + \frac{d(d-3)}{4}.
\end{equation}
\end{coro}

\begin{proof}
Fix $d$ and we write $g=g(d)$ for notational convenience. When $\chi\ge g$, the statement is vacuous by part (1) of Theorem \ref{maintheorem}. Hence we may assume $1-g\le \chi\le g-1$. In particular, we may assume $d\ge 3$. Let $\beta(\chi)$ be the bound in part (2) of Theorem \ref{maintheorem} and $\gamma(\chi)$ be the bound in the statement of this corollary. We claim that if $1-g\le \chi\le g-1$, then we have $\beta(\chi)\le \gamma(\chi)$.

  It is easy to check that $\beta(1-g)=\gamma(1-g)$ and $\beta(g-1)=\gamma(g-1)$.

  Then we can check
  \[ \beta(\chi) = \beta(g-1) - \#\{\chi' \colon \chi\le\chi'< g-1, k'-d+r'+2\ge 0 \},\]
  where $k'$ and $r'$ are obtained from $\chi'$ similarly as in Theorem \ref{maintheorem}, and
  \[ \gamma(\chi) = \gamma(g-1) - \frac{1}{2}\#\{\chi' \colon \chi\le\chi'< g-1 \}.\]
  We claim that more than half of $\chi'$ with $\chi\le\chi'<g-1$ yield $k'-d+r'+2\ge0$. Hence we have $\beta(\chi)\le\gamma(\chi)$.

  By definition, $\chi$ is in one-to-one correspondence with $(k,r)$. For $\chi'$ such that $\chi\le\chi'<g-1$, corresponding $(k',r')$ are those with $k<k'\le d-4$ or with $k'=k$ and $r'\ge r$. The condition $k'-d+r'+2\ge0$ requires $(k',r')$ to be above or on the line $k+r=d-2$, which is satisfied more than half of such points. So, the corollary follows.
  Note that when $\chi=1-g$, exactly half of $\chi'$ with $\chi\le\chi'<g-1$ satisfy $k'-d+r'+2\ge0$, which gives us $\beta(1-g)=\gamma(1-g)$.
\end{proof}

This corollary is conjectured in \cite[\S1.4]{maican} for $X=\PP^2$ and $0\le \chi<d$. We have proved a more general statement.
\begin{rema}
  From the proof of Corollary \ref{conj1} and Theorem \ref{maintheorem}, the equality in \eqref{clifford} holds if and only if $F$ is a plane sheaf and its spectrum is either $[0,-1,\cdots, 1-d]$ or $[d-3,d-2,\cdots, -2]$, which is respectively when $F\simeq \cO_C$ or $F\simeq \omega_C$ for some plane curve $C$. We note that the bound in \eqref{clifford} is produced from the Clifford theorem by applying the Riemann-Roch theorem using the arithmetic genus of the plane curve. Hence for non-planar sheaves, Corollary \ref{conj1} can be improved according to the arithmetic genus of the curve on which $F$ is supported.
\end{rema}

\section{Stratification of $\bM(d,\chi)$ via spectra}\label{sec:deep}
In this section, we return to the case $X=\PP^2$. We study the stratification of $\bM(d,\chi)$ with respect to the spectrum. It turns out that the spectrum we chose in \S \ref{sec2} corresponds to the closed stratum with the biggest codimension (cf. Proposition \ref{katz}). In Proposition \ref{prop:dsodd} and Proposition \ref{prop:dseven}, we describe this locus in terms of relative Hilbert schemes by using the wall-crossing technique of \cite{kiryong}.

Let $v$ be a spectrum. We denote by $\bM_v(d,\chi)$ the locus in $\bM(d,\chi)$ consisting of sheaves having spectrum $v$.
\begin{lem}
  $\{\bM_v(d,\chi)\}$ is a finite locally closed stratification of $\bM(d,\chi)$.
\end{lem}
\begin{proof}
 As noted in Remark \ref{rem1}, the spectrum of a sheaf $F$ is completely determined by $k\mapsto h^0(F(k))$. Hence the subscheme $\bM_v(d,\chi)$ is cut out by conditions on $h^0(F(k))$. By the semicontinuity theorem, it is clear that $\bM_v(d,\chi)$ is locally closed. Since only finitely many spectra are possible after fixing $d$ and $\chi$, $\bM_v(d,\chi)$ is a finite locally closed stratification of $\bM(d,\chi)$.
\end{proof}

The strata provided by our classification are irreducible varieties.
\begin{prop}\label{katz}
 $\bM_v(d,\chi)$ is an irreducible variety of $\bM(d,\chi)$ whose codimension is the size of the set $ \{(a_i, a_j)\colon |a_i-a_j|\ge 2\}$.
\end{prop}
\begin{proof}
Fix a point $a \in \PP^2$. As in \S\ref{sec2}, let $$U_a=\{F\in \bM_v(d,\chi)\colon a\notin \mbox{Supp}(F)\}.$$
We first claim that it is enough to show that
$U_{a,v}:= \bM_v(d,\chi) \cap U_a$ is an irreducible variety.
Consider $$R_v=\{(F,p)\in \bM_v(d,\chi)\times \PP^2 \colon p\notin F\}.$$ Then the projection $R_v\to \bM_v(d,\chi)$ is surjective. We also have the surjective morphism
\begin{center}
\begin{tabular}{ccc}
  $PGL_3\times U_{a,v}$ &$\to$& $R_v$\\
  $(\alpha, F)$ &$\mapsto$& $(\alpha^*F, \alpha^{-1} a)$
\end{tabular}.
\end{center}
where $\alpha\in PGL_3$ is considered as an automorphism $\alpha\colon \PP^2\to \PP^2$. Hence $\bM(d,\chi)$ is irreducible provided that $U_{a,v}$ is irreducible.\footnote{This argument is due to S. Katz.}

To prove the irreducibility of $U_{a,v}$, we check the semistability of Higgs bundles is an open condition. We follow the argument in the proof of \cite[Proposition 2.3.1]{HL}. Let $G$ be the vector bundle on $\pp^1$ corresponding to the spectrum $v$. Let $S:=\Hom(G,G(1))$ be the affine space.
Since the relative Quot scheme is a fine moduli space, there exist a universal family $\cF$ of Higgs bundles parameterized by $S$ and a projective morphism:
$$
\pi:Q(d' ,\chi' )=\mbox{Quot}^{d' ,\chi'}(\cF/S)\lr S
$$
which parameterizes the couples $(\phi, F'=(G',\phi'))$ such that $\phi\in \Hom(G,G(1))$ and the quotient $(G,\phi)\twoheadrightarrow F'$ (For detail, see the page 651 of \cite{lepot1}).
Since $\pi$ is projective, the image $\pi(Q(d',\chi'))$ is a closed subset in $S$.
Also, since the number of the destabilizing pair $(d',\chi')$ is finite,  the union of such $\pi(Q(d',\chi'))$ is also closed in $S$. Therefore we have an open subvariety $\Hom(G,G(1))^{ss}$ in the affine space $S$.
Hence from the quotient map
$$
\Hom(G,G(1))^{ss}\lr U_{a,v},
$$
one can say that $U_{a,v}$ is an irreducible variety.

On the other hand, from the proof of Proposition 3.14 in \cite{lepot1}, the codimension of $\bM_v(d,\chi)$ in $\bM(d,\chi)$ is
$$
\mbox{dim}\Ext^1(G,G)-\mbox{dim}\Ext^1(G,G(1))
$$
and hence is equal to the size of the set $ \{(a_i, a_j)\colon |a_i-a_j|\ge 2\}$.
\end{proof}

\begin{exam}
In the series of papers \cite{maican, maican1, maican2}, Maican et al have studied the moduli spaces $\bM(d,\chi)$ for $d\le 6$ using a stratification. They have classified sheaves $F$ in $\bM(d,\chi)$ by conditions on $h^i(F(j))$ and $h^0(F\otimes \Omega^1(1))$, which in turn determine the syzygy types of the sheaves $F$. In many cases, their stratification coincides with our stratification by spectra. For example, $\bM(4,1)$ is a union of $X_0$ and $X_1$ in \cite{maican}. In our notation, $X_0$ is $\bM_{[0,-1,-1,-1]}(4,1)$ and $X_1$ is $\bM_{[1,0,-1,-1]}(4,1)$. However, since they have used the additional conditions on $h^0(F\otimes \Omega^1(1))$, their stratification is finer than the stratification by spectra.
\end{exam}
This stratification by spectra becomes very complicated as $d$ increases. In this paper, we confine ourselves to the deepest stratum.
\begin{defi}
\begin{enumerate}
  \item Let $v=[a_1,\cdots, a_d]$ be a spectrum. We call $v$ the \emph{deepest spectrum} if it satisfies the condition in Lemma \ref{lem:maxspectrum}.
  \item When $v$ is the unique deepest spectrum for fixed $d$ and $\chi$, we denote $\bM_v(d,\chi)$ by $\ds$.
\end{enumerate}
\end{defi}

\begin{prop}
 $\ds$ is the unique closed stratum in $\{\bM_v(d,\chi)\}$ having the biggest codimension among strata.
\end{prop}
\begin{proof}
We note that $v=[a_1,\cdots, a_d]$ is the deepest spectrum if and only if $a_1$ is maximal and $a_d$ is minimal among all spectra for $\bM(d,\chi)$.
Hence we have
\[ \ds=\{G\in \bM(d,\chi)\colon h^0(G(-a_1))>0 \text{ and } h^1(G(-a_d-2))>0 \}.\]
Thus, it is closed by the semicontinuity theorem. Moreover, it is clear from Proposition \ref{katz} that $\ds$ has the biggest codimension.
\end{proof}

Hence we see that $\ds$ is indeed the deepest stratum in the stratification $\{\bM_v(d,\chi)\}$.

\begin{prop}\label{prop:dual}
\begin{enumerate}
  \item $\ds\simeq \ds[d+\chi]$.
  \item If $0\le \chi < d$, then $\ds\simeq \ds[d-\chi]$.
\end{enumerate}
\end{prop}
\begin{proof}
The first statement is obtained by twisting by $\opt(1)$. For the second statement, by part (2) of Example \ref{lem:dualspectrum}, under the isomorphism $\bM(d,\chi)\simeq \bM(d,d-\chi)$ sending $F$ to $F^D(1)$, sheaves in $\ds$ correspond to sheaves in $\ds[d-\chi]$.
\end{proof}

Hence it is enough to consider only finitely many cases for fixed $d$. In what follows, we describe $\ds$ as a subscheme of the relative Hilbert scheme. 
We start with reviewing results in \cite{kiryong}.
\begin{defi}
We denote by $\bB(d,n)$ the relative Hilbert scheme consisting of pairs $(C,Z)$ where $C$ is degree $d$ curve in $\PP^2$ and $Z\subset C$ is length $n$ subscheme.
\end{defi}

In \cite{kiryong}, we study the relationship between $\bB(d,n)$ and the moduli space of stable sheaves using $\alpha$-stable pairs.
A pair $(s, F)$ consists of a sheaf $F$ on $\PP^2$ and one-dimensional subspace $s \subset H^0(F)$. Given a positive number $\alpha$, a pair $(s,F)$ is called \emph{$\alpha$-semistable} if $F$ is pure and for any subsheaves $F'\subset  F$, the inequality
$$
\frac{\chi(F'(m))+\delta\cdot\alpha}{r(F')} \leq \frac{\chi(F(m))+\alpha}{r(F)}
$$
holds for $m\gg 0$. Here $r(F)$ is the leading coefficient of Hilbert polynomial $\chi(F(m))$ and $\delta=1$ if the section $s$ factors through $F'$ and $\delta=0$ otherwise.
When the strict inequality holds, $(s,F)$ is called $\alpha$-stable. By the work of Le Potier \cite{lepot2}, there exist the moduli spaces $\bM^{\alpha}(d,\chi)$ which parameterizes S-equivalent classes\footnote{By definition, two semistable pairs are S-equivalent if they have isomorphic Jordan-H\"{o}lder filtration.} of $\alpha$-semistable pairs with Hilbert polynomial $dm+\chi$.

At the extreme values of $\alpha$, the moduli spaces of $\alpha$-stable pairs are related to $\bB(d,n)$ and $\bM(d,\chi)$.
\begin{enumerate}
\item If the stability parameter $\alpha$ is sufficiently large (denoted by $\alpha=\infty$), this moduli space $\bM^{\alpha}(d,\chi)$ is isomorphic to the relative Hilbert scheme $\bB(d,n)$ where $n$ is $\chi-1+g(d)$. The correspondence is the following \cite[Proposition B.8]{pt3}: For $\infty$-stable pair $(s,F)$,  the section $s:\cO_{\PP^2}\lr F$ induces a short exact sequence $$\ses{\cO_C}{F}{Q}$$ for a degree $d$ curve $C$ and zero dimensional sheaf $Q$ with length $n$. As taking dual  $\mathcal{H}om_C(-, \cO_C)$ to the short exact sequence, we obtain a zero dimensional subscheme $Z$ defined by the surjection: $$\cO_C\twoheadrightarrow  \mathcal{E}xt^1_C(Q, \cO_C),$$
where the ideal sheaf is given by $I_{Z,C}=\mathcal{H}om_C(F, \cO_C)$.

\item On the other hand, if the stability parameter is sufficiently small (denoted by $\alpha=0^+$), then there exists a forgetting morphism into the space $\bM(d,\chi)$, denoted by
$$
\xi:\bM^{0^+}(d,\chi)\lr \bM(d,\chi).
$$
\end{enumerate}
The wall-crossing behavior of the moduli spaces $\bM^{\alpha}(d,\chi)$ is studied in \cite{mhe,kiryong}. The moduli space changes only finitely many values of $\alpha$, which are called \emph{walls}. The wall occurs if there exist strictly $\alpha$-semistable pairs.
If there is no wall between $\bM^{\infty}(d,\chi)$ and $\bM^{0^+}(d,\chi)$, two spaces are isomorphic.

\begin{rema}
Let us denote by $\bM_v^\alpha(d,\chi)$ the subscheme of $\bM^\alpha(d,\chi)$ consisting of pairs $(s,F)$ where the sheaf $F$ has the spectrum $v$. Then one can also consider the wall-crossing behavior of $\bM_v^\alpha(d,\chi)$ for a fixed $v\colon$ the spectrum of a sheaf remains unchanged after the elementary modification of pairs \cite{kiryong}.

We also remark that strictly semistable sheaves exist in $\bM_v(d,\chi)$ if and only if $\alpha=0$ is a ``wall'' for $\bM_v^{\alpha}(d,\chi)$, in which case the forgetting morphism $\xi\colon\bM_v^{0^+}(d,\chi)\to \bM_v(d,\chi)$ is a projective bundle. (cf. \cite[Lemma 4.2.2]{choi}) Furthermore, if $h^0(F)=1$ for any $(s,F)\in \bM_v(d,\chi)$, then $\xi$ induces an isomorphism $\bM_v^{0^+}(d,\chi)\simeq \bM_v(d,\chi)$. We will use this fact repeatedly.
\end{rema}

\begin{prop}\label{prop:dsodd}
Suppose $d$ is odd and $0\le\chi<\frac{d}{2}$. Then $\ds$ is isomorphic to the subscheme of $\bB(d,\chi)$ consisting of pairs $(C,Z)$ where $C$ is a degree $d$ curve and $Z$ is colinear $\chi$ points on $C$.
\end{prop}
\begin{proof}
Let $v=[a_1,\cdots, a_d]$ be the corresponding deepest spectrum. From the condition, it is easy to check that $a_1=\frac{d-3}{2}$ and $a_1>a_2$. We also have $a_d=-\frac{d+1}{2}$ if $\chi=0$ and $a_d=-\frac{d-1}{2}$ if $\chi>0$. When $\chi=0$, the same reasoning as below shows that the sheaves in $\ds$ is of the form $\cO_C$ for some degree $d$ curve $C$. Hence $\ds[0]\simeq \bB(d,0)$ as required.

Assume $\chi>0$ so that $a_d=-\frac{d-1}{2}$. Let $\chi'= \chi-d(\frac{d-3}{2})$ and $v'=[a_1-\frac{d-3}{2},\cdots, a_d-\frac{d-3}{2}]$. Then by twisting by $-\frac{d-3}{2}$ we have
\[\ds\simeq \ds[\chi'] =\{F\in \bM(d,\chi')\colon h^0(F)=1 \text{ and } h^1(F(d-4))>0\}.  \]
Therefore, $\ds\simeq \bM_{v'}^{0^+}(d,\chi')$. We claim that no wall crossing is necessary between $\bM^{0^+}(d,\chi')$ and $\bM^{\infty}(d,\chi')$. Indeed, if there is a wall, the splitting type of wall is
\[(1, dm+\chi') = (1, (d-e)m + 1-g(d-e)+z) + (0, em + \chi' -z-1+ g(d-e)), \]
for an integer $1\le e<d$ and a nonnegative integer $z$.
Then corresponding stability parameter $\alpha$ is given by
\[ \frac{\chi'+\alpha}{d} = \frac{1}{e} (\chi' -z-1+ g(d-e)).  \]
After simplifying, we get
\[\alpha= (e-d)(\frac{d}{2}-\frac{\chi}{e}) - \frac{d}{e}z. \]
So, if $\chi<\frac{d}{2}$, $\alpha$ is negative for any $e$ and $z$. Hence, there is no wall and we have $\ds\simeq \bM_{v'}^{\infty}(d,\chi')$.

Under isomorphism $\bM^{\infty}(d,\chi')\cong\bB(d,\chi)$ described above, a stable pair $(s,F)$ in $\bM^{\infty}(d,\chi')$ corresponds to a pair of its support $C$ and a length $\chi$ subscheme $Z$ of $C$ whose structure sheaf is given by $\mathcal{E}xt^1(Q, \cO_C)$. We recall that $I_{Z,C}=\mathcal{H}om_C(F, \cO_C)$.

We now show that the condition $h^1(F(d-4))>0$ is equivalent to the condition $Z$ being colinear. By Serre duality on $C$, we have $$ H^1(F(d-4))^*= \Hom_C(F(d-4), \cO_C(d-3))=\Hom_C(F, \cO_C(1)). $$
Therefore, $h^1(F(d-4))>0$ if and only if
\[0\ne \Hom_C(F,\cO_C(1) )=H^0(I_{Z,C}(1))=H^0(I_{Z,\PP^2}(1)), \]
or equivalently, $Z$ is colinear.
\end{proof}

\begin{rema}
  For $\frac{d}{2}<\chi<d$, the same proof works except for the case $\chi=d-1$. If $\chi=d-1$, we have $a_1=a_2$ so that $h^0(F)=2$ for $F\in \bM(d,\chi')$ and hence $\bM_v^{0^+}(d,\chi')\simeq \bM_v^{\infty}(d,\chi')$ is  a $\PP^1$-bundle over $\ds$. Geometrically, this means a pair of $(d-1)$-tuples of colinear points is identified if $d$-th points determined by the line and the curve coincide. This is consistent with Proposition \ref{prop:dual}, that is, the subscheme of $\bB(d,\chi)$ having colinear points is isomorphic to the subscheme of $\bB(d,d-\chi)$ having colinear points unless $\chi= 1 \text{ or } d-1$.
\end{rema}

\begin{prop}\label{prop:dseven}
Suppose $d$ is even and $\frac{d}{2}\le \chi<d$. Then $\ds$ is isomorphic to the subscheme of $\bB(d,\chi-\frac{d}{2})$ consisting of pairs $(C,Z)$ where $C$ is a degree $d$ curve and $Z$ is colinear $\chi-\frac{d}{2}$ points on $C$.
\end{prop}

\begin{proof}
The proof is parallel to Proposition \ref{prop:dsodd}. Note that the case $\chi=d$ is dropped, which we will show is the only case that we have to consider wall-crossing.

Let $v=[a_1,\cdots, a_d]$ be the corresponding deepest spectrum as before. Then we have $a_1=\frac{d-2}{2}$ and $a_1>a_2$. We also have $a_d=-\frac{d}{2}$ if $\chi=\frac{d}{2}$ and $a_d=-\frac{d-2}{2}$ if $\chi>\frac{d}{2}$.

%

The same argument as in Proposition \ref{prop:dsodd} works with $\chi'=\chi-d(\frac{d-2}{2})$. We only have to check that no wall-crossing occurs. 
Similarly as before if there is a wall, the splitting type of the wall is
\[(1, dm+\chi') = (1, (d-e)m+ 1-g(d-e)+z) + (0, em + \chi' -z-1+ g(d-e)), \]
for an integer $1\le e<d$ and a nonnegative integer $z$.
Then corresponding stability parameter $\alpha$ is given by
\[ \frac{\chi'+\alpha}{d} = \frac{1}{e} (\chi' -z-1+ g(d-e)).  \]
After simplifying, we get
\[\alpha= (e-d)(d(\frac{e+1}{2e})-\frac{\chi}{e}) - \frac{d}{e}z. \]
Since $1\le e<d$ and $\chi<d$, $\alpha$ is negative. Hence, there is no wall.
\end{proof}

\begin{rema}
By direct computation of the dimension, one can easily check that the codimension of $\ds$ in $\bM(d, \chi)$ is
$\frac{d^2-3d-2}{2}$, which is the  largest one among all spectra. (cf. \cite[Proposition 3.14]{lepot1})
\end{rema}


Note that in the proof of Proposition \ref{prop:dseven}, we have a wall $\alpha=0$ when $\chi=d$, $e=1$, and $z=0$. This is because of the presence of strictly semistable sheaves.

\begin{coro}\label{cor:even0}
  Let $0\le \chi<d$. Unless $d$ is even and $\chi=0$, there are no strictly semistable sheaves in $\ds$. If $d$ is even and $\chi=0$, the locus $\ds[0]^s$ of stable sheaves in $\ds[0]$ is isomorphic to the subscheme of relative Hilbert scheme $\bB(d,\frac{d}{2})$ consisting of $(C,Z)$ where $Z$ is colinear and the line containing $Z$ is not a component of $C$.
\end{coro}

\begin{proof}
   The strictly semistable sheaves exist if and only if $\alpha=0$ becomes a wall. In the proof of Proposition \ref{prop:dsodd} and \ref{prop:dseven}, we have a wall at $\alpha=0$ only when $d$ is even and $\chi=0$. Since $e=1$ and $z=0$ is the only possibility, a pair $(s,F)$ in $\ds[0]$ is strictly semistable only if its support is a union of degree $(d-1)$ curve $C'$ and a line $L$ and the section $s$ is taken from the structure sheaf $\cO_{C'}$, so that the cokernel of $s$ is supported on $L$.
\end{proof}
This corollary generalizes the description of the deepest stratum of $\bM(4,0)$ and $\bM(6,0)$ in \cite{maican, maican2}.
\bibliographystyle{amsplain}

\begin{thebibliography}{99}
\bibitem{choi} J. Choi. {\em Enumerative invariants for local Calabi-Yau threefolds}. Ph.D. Thesis, University of Illinois, 2012.
\bibitem{choip1} J. Choi. {\em Genus zero BPS invariants for local $\mathbb{P}^1$}. Int. Math. Res. Not., Advance Access published October 15, 2012, doi: 10.1093/imrn/rns225.
\bibitem{kiryong} J. Choi and K. Chung. {\em Moduli spaces of $\alpha$-stable pairs and wall-crossing on $\mathbb{P}^2$}. arXiv:1210.2499.
\bibitem{maican} J.-M. Dr\'ezet and M. Maican. {\em On the geometry of the moduli spaces of semi-stable sheaves supported on plane quartics}. Geometriae Dedicata DOI 10.1007/s10711-010-9544-1.
\bibitem{fredu} H. Freiermuth and G. Trautmann. {\em On the moduli scheme of stable sheaves supported on cubic space curves}. Amer. J. Math. {\bf 126} (2004), no. 2, 363-393.
\bibitem{har} R. Hartshorne. {\em Generalized divisors on Gorenstein curves and a theorem of Noether}. J. Math. Kyoto Univ. {\bf 26} (3) (1986), 375-386.
\bibitem{mhe} M. He. {\em Espaces de Modules de syst\`emes coh\'erents}. Internat. J. of Math. {\bf 7} (1998), 545-598.
\bibitem{HL} D. Huybrechts and M. Lehn. {\em The geometry of moduli spaces of sheaves}. Aspects of Mathematics, E31. Friedr. Vieweg Sohn, Braunschweig, 1997.
\bibitem{kool} M. Kool. {\em Fixed point loci of moduli spaces of sheaves on toric varieties}. Adv. Math. {\bf 227} (2011), no. 4, 1700-1755.
\bibitem{lepot1} J. Le Potier. {\em Faisceaux semi-stables de dimension $1$ sur le plan projectif}. Rev. Roumanine Math. Appl., {\bf 38} (1993), 7-8, 635-678.
\bibitem{lepot2} J. Le Potier. {\em Syst\`emes coh\'erents et structures de niveau}. Ast\'erisque 214, {\bf 143}, 1993.
\bibitem{maican1} M. Maican. {\em On the moduli spaces of semi-stable plane sheaves of dimension one and multiplicity five.} To appear in Illinois J. of Math. arXiv:0910.5327.
\bibitem{maican2} M. Maican. {\em The classification of semi-stable plane sheaves supported on sextic curves.} To appear in Kyoto J. of Math. arXiv:1205.0278.
\bibitem{maicandual} M. Maican. {\em A duality result for moduli spaces of semistable sheaves supported on projective curves.} Rendiconti del Seminario Matematico della Universit\`{a} di Padova {\bf 123} (2010), 55-68.
\bibitem{moz} S. Mozgovoy. {\em Solutions of the motivic ADHM recursion formula.} arXiv:1104.5698.
\bibitem{pt3} R. Pandharipande and R. Thomas. {\em Stable pairs and BPS invariants}. J. Amer. Math. Soc. {\bf 23} (2010), no.1, 267-297.
\bibitem{rayan} S. Rayan. {\em Geometry of co-Higgs bundles}. Ph.D. Thesis, Oxford, 2011.
\bibitem{simp} C. Simpson. {\em Moduli of representations of the fundamental group of a smooth projective variety}. I. Inst. Hautes Etudes Sci. Publ. Math. No. {\bf 79} (1994), 47-129.
\end{thebibliography}

\end{document}